\documentclass[12 pt]{amsart}

\newtheorem{theorem}{Theorem}[section]

\newtheorem{proposition}[theorem]{Proposition}
\newtheorem{corollary}[theorem]{Corollary}

\theoremstyle{definition}

\usepackage{amscd,amssymb}
\begin{document}

\title[Dimension of automorphisms with fixed degree]
{Dimension of automorphisms with fixed degree for polynomial
algebras}

\author[Vesselin Drensky and Jie-Tai Yu]
{Vesselin Drensky and Jie-Tai Yu}
\address{Institute of Mathematics and Informatics,
Bulgarian Academy of Sciences,
1113 Sofia, Bulgaria}
\email{drensky@math.bas.bg}
%\address{Institute
%of Pure Mathematical Research, Yunnan Normal University, Kunming,
%China; and}
\address{Department of Mathematics, The University of Hong Kong, Hong
Kong SAR, China} \email{yujt@hkucc.hku.hk}

\thanks
{The research of Vesselin Drensky was partially supported by Grant
MI-1503/2005 of the Bulgarian National Science Fund.}

\thanks
{The research of Jie-Tai Yu was partially supported by an RGC-CERG grant.}


\begin{abstract}
Let $K[x,y]$ be the polynomial algebra in two variables over an
algebraically closed field $K$. We generalize to the case of any
characteristic the result of Furter that over a field of
characteristic zero the set of automorphisms $(f,g)$ of $K[x,y]$
such that $\max\{\text{deg}(f),\text{deg}(g)\}=n\geq 2$ is
constructible with dimension  $n+6$. The same result holds for the
automorphisms of the free associative algebra $K\langle x,y\rangle$.
We have also obtained analogues for free algebras with two
generators in Nielsen -- Schreier varieties of algebras.
\end{abstract}

\maketitle

\section*{Introduction}

Our paper is inspired by the following problem. {\it Let $K$ be an
arbitrary field of any characteristic and let $K[x,y]$ be the
polynomial algebra in two variables. How many $K$-automorphisms
$\varphi=(f,g)$ of $K[x,y]$   satisfying
$\deg(\varphi):=\max\{\deg(f),\deg(g)\}=n$?} Here $\varphi=(f,g)$
means that $f=\varphi(x)$, $g=\varphi(y)$.

In the sequel all automorphisms are $K$-automorphisms. Motivated by
the problem of Arnaud Bodin \cite{B} to determine the number of
automorphisms $\varphi$ with $\deg(\varphi)\leq n$ over the finite
field ${\mathbb F}_q$ with $q$ elements, in our recent paper
\cite{DY} we determined the number $p_n$ of ${\mathbb
F}_q$-automorphisms of degree $n$ and found the Dirichlet series
generating function of the sequence $p_n$, $n\geq 1$.

When the field $K$ is infinite, the natural translation of the
problem is in the language of algebraic geometry. When the field $K$
is algebraically closed and of characteristic 0 this was considered
by Bass, Connell and Wright \cite{BCW} as a possible approach to the
Jacobian conjecture. Identifying the endomorphisms
$\varphi=(f_1,\ldots,f_m)$ of degree $\leq n$ of
$K[X]=K[x_1,\ldots,x_m]$ with points of $K^N$ where
$N=m\binom{m+n}{m}$ is $m$ times the number of monomials of degree
$\leq n$ in $K[X]$, the set ${\mathcal J}_{1,n}$ of all
endomorphisms of degree $\leq n$ with Jacobian $J(f_1,\ldots,f_m)=1$
is a closed subvariety of $K^N$. The subset of automorphisms
${\mathcal A}_{1,n}$ in ${\mathcal J}_{1,n}$ is also a closed
subvariety of $K^N$. In \cite{BCW} it is proved that the Jacobian
conjecture would follow if ${\mathcal J}_{1,n}$ is irreducible and
$\dim({\mathcal A}_{1,n})=\dim({\mathcal J}_{1,n})$. Furter \cite{F}
studied in detail the case of two variables. Corollary 1.6 of
\cite{BCW} gives that the set ${\mathcal A}_n$ of all automorphisms
of degree $\leq n$ is an algebraic variety. Furter showed that for
$K[x,y]$ the variety ${\mathcal A}_n$ is of dimension $n+6$ for any
$n\geq 2$ (and of dimension 6 for $n=1$). He also proved that
unfortunately, the variety ${\mathcal A}_n$ is irreducible for
$n\leq 3$ only, that means the approach for the Jacobian conjecture
suggested by \cite{BCW} does not work.  Applying the result of Moh
\cite{M} that the Jacobian conjecture is true for all endomorphisms
of $K[x,y]$ with invertible Jacobian and of degree $n\leq 100$,
Furter obtained that ${\mathcal J}_n$ is reducible for all
$n=4,\ldots,100$. His proof uses the theorem of Jung -- van der Kulk
\cite{J, K} that the automorphisms of the polynomial algebra
$K[x,y]$ over any field $K$ are tame and the structure of the
automorphism group
\[
\text{Aut}(K[x,y]) = A\ast_CB,\quad C=A\cap B,
\]
where $A\ast_CB$ is the free product of the subgroup
$A$ of affine automorphisms and and the subgroup $B$ of triangular automorphisms
with amalgamated subgroup $C=A\cap B$.

Our first theorem transfers the result of Furter \cite{F} on the
dimension of ${\mathcal A}_n$ to the case of any algebraically
closed field of arbitrary characteristic. It requires a basic
knowledge of algebraic geometry only. As in \cite{F} and \cite{DY}
we use the theorem of Jung -- van der Kulk and the decomposition
$\text{Aut}(K[x,y]) = A\ast_CB$ and obtain that {\it the set
${\mathcal A}^{(n)}$ of all automorphisms $\varphi$ of $K[x,y]$ with
$\deg(\varphi)=n\geq 2$ is a constructible subset of dimension $n+6$
in $K^N$.} As an immediate consequence we obtain that the dimension
of the set ${\mathcal A}_n$ of automorphisms of degree $\leq n$ is
also equal to $n+6$ for all $n\geq 2$. We think that, although using
similar ideas, our proof is simpler than that of Furter \cite{F}.

We are also able to obtain that the dimension of the set of
coordinates in $K[x, y]$ with a fixed degree $n$\  $(n\ge 2)$ is
equal to $n+3$.

By the theorem of Czerniakiewicz and Makar-Limanov \cite{Cz, ML} for
the tameness of the automorphisms of $K\langle x,y\rangle$ over any
field $K$ and the isomorphism $\text{Aut}(K[x,y])\cong
\text{Aut}(K\langle x,y\rangle)$ which preserves the degree of the
automorphisms we derive immediately that the dimension of the
automorphisms of degree $n$ of the free associative algebra
$K\langle x,y\rangle$ is also $n+6$. Here $(f,g)$ is identified with
a point of $K^p=2(1+2+2^2+\cdots+2^n)=2^{n+2}-2$, because $2^m$ is
the dimension of the vector space of homogeneous polynomials of
degree $m$ in $K\langle x,y\rangle$. In this case the commutator
criterion of Dicks \cite{D} gives that $\varphi=(f,g)$ is an
automorphism if and only if $[f,g]=fg-gf=a[x,y]$ for some nonzero
constant $a$.

As in \cite{DY} we have an analogue of our result on the
automorphisms of $K[x,y]$ in the case of free Nielsen -- Schreier
algebras $F(x,y)$ with two generators, i.e., if the subalgebras of
$F(x,y)$ are free in the same class of algebras. For the exact
result we make the additional requirement that the dimension $c_n$
of homogeneous polynomials of degree $n$ in $F(x)$ satisfies
$c_n\leq c_{n+1}$ for any $n\geq 1$. Examples of such algebras are
the free nonassociative algebras and free commutative nonassociative
algebras. {\it The dimension of the set of automorphisms of degree
$n$ of $F(x,y)$ is
\[
d_n=c_1+c_2+\cdots+c_n+4+2\varepsilon,
\]
where $\varepsilon=1$ for unitary algebras and $\varepsilon=0$ for
nonunitary algebras.} In the case of the free nonassociative algebra
$K\{x,y\}$ we have obtained the generating function (see, for
instance, \cite{Wi}) of the sequence $d_n$:
\[
d(t)=\sum_{n\geq 1}d_nt^n=\frac{1+4(2+\varepsilon)t-\sqrt{1-4t}}{2(1-t)}.
\]

\

\section{Dimension of automorphisms of polynomial algebras}

Using that the automorphisms of $K[x,y]$ are tame and the structure of
$\text{Aut}(K[x,y])=A\ast_CB$ of the automorphism group of $K[x,y]$
it is easy to obtain the following well known fact, see e.g.
\cite{W} or \cite{F}. A simple proof is given in \cite{DY}.
Below we write the automorphisms as functions. If
$\varphi=(f_1(x,y),g_1(x,y))$, $\psi=(f_2(x,y),g_2(x,y))$, then
$\varphi\circ\psi(u)=\varphi(\psi(u))$, $u\in K[x,y]$, and hence
\[
\varphi\circ\psi=(f_2(f_1(x,y),g_1(x,y)),g_2(f_1(x,y),g_1(x,y))).
\]

\begin{proposition}\label{canonical form of automorphisms}
Define the sets of automorphisms of $K[x,y]$
\[
A_0=\{\iota=(x,y),\quad \alpha=(y,x+ay)\mid a\in K\},
\]
\[
B_0=\{\beta=(x+h(y),y)\mid h(y)\in y^2K[y]\}.
\]
Every automorphism $\varphi$ of $K[x,y]$ can be presented in a unique
way as a composition
\[
\varphi=(f,g)=\alpha_1\circ\beta_1\circ\alpha_2\circ\beta_2\circ
\cdots\circ\alpha_k\circ\beta_k\circ\lambda,
\]
where $\alpha_i\in A_0$, $\alpha_2,\ldots,\alpha_k\not=\iota$,
$\beta_i\in B_0$, $\beta_1,\ldots,\beta_k\not=\iota$, $\lambda\in A$.
If $\beta_i=(x+h_i(y),y)$ and $\text{\rm deg}(h_i(y))=n_i$,
then the degree of $\varphi$
\[
n=\text{\rm deg}(\varphi)=\max\{\text{\rm deg}(f),\text{\rm deg}(g)\}=n_1\cdots n_k
\]
is equal to the product of the degrees of $\beta_i$.
\end{proposition}

Now we assume that the affine spaces $K^N$ are equipped with the Zariski topology.
Recall that the set ${\mathcal W}\subseteq K^N$ is constructible if it is a finite union
of intersections ${\mathcal U}_i\cap {\mathcal V}_i$,
where ${\mathcal U}_i$ is open and ${\mathcal V}_i$
is closed in the Zariski topology.
We order the monomials of degree $\leq n$ in $K[x,y]$:
$u_1,\ldots,u_{N/2}$, $N=(n+1)(n+2)$, and identify the endomorphism
$\varphi=(f,g)$ of $K[x,y]$ of degree $\leq n$,
\[
f=\sum_{j=1}^{N/2}a_ju^j,\quad g=\sum_{j=1}^{N/2}b_ju^j,\quad a_j,b_j\in K,
\]
with the point
\[
\pi(\varphi)=(a_1,\ldots,a_{N/2},b_1,\ldots,b_{N/2})\in K^N.
\]
Clearly, the map $\pi$ is a bijection of the set of endomorphisms
of degree $\leq n$ onto $K^N$.

The following theorem is a generalization  of the main results of
Furter \cite{F}, who proved the zero characteristic case.

\begin{theorem}\label{dimension of automorphisms}
Let ${\mathcal A}^{(n)}$ be the set of all automorphisms of degree $n$
of $K[x,y]$. Then the set $\pi({\mathcal A}^{(n)})$ is a constructible
subset of $K^N$ of dimension $d_n=n+6$, $n\geq 2$.
\end{theorem}

\begin{proof}
Let $n\geq 2$. We identify the affine endomorphism
\[
\lambda=(a_1x+b_1y+c_1,a_2x+b_2y+c_2)
\]
with the point
$\pi(\lambda)=(a_1,b_1,c_1,a_2,b_2,c_2)\in K^6$, the automorphism
\[
\alpha=(y,x+ay)\in A_0
\]
with the point $\pi_0(\alpha)=a\in K$,
and the triangular automorphism
\[
\beta=(x+h(y),y),\quad h(y)=h_my^m+h_{m-1}y^{m-1}+\cdots+h_2y^2,
\]
with the point $\pi_m(\beta)=(h_m,h_{m-1},\ldots,h_2)\in K^{m-1}$.

We fix an ordered factorization
$n=n_1\cdots n_k$, $n_i\geq 2$.
We denote by ${\mathcal B}(n_1,\ldots,n_k)$ the set of all automorphisms
\[
\varphi=(f,g)=\alpha_1\circ\beta_1\circ\alpha_2\circ\beta_2\circ
\cdots\circ\alpha_k\circ\beta_k\circ\lambda
\]
with $\deg(h_i)=n_i$, $\alpha_1\not=\iota$ and by ${\mathcal
C}(n_1,\ldots,n_k)$ the corresponding set with $\alpha_1=\iota$. We
define a bijection $\vartheta$ between the endomorphisms
\[
\varphi=\alpha_1\circ\beta_1\circ\alpha_2\circ\beta_2\circ
\cdots\circ\alpha_k\circ\beta_k\circ\lambda,
\]
$\alpha_1\not=\iota$, $\lambda$ being an affine endomorphism,
and the points
\[
\vartheta(\varphi)=(\pi_0(\alpha_1),\ldots,\pi_0(\alpha_k),
\pi_{n_1}(\beta_1),\ldots,\pi_{n_k}(\beta_k),\pi(\lambda))\in K^M,
\]
\[
M=k+(n_1-1)+\cdots+(n_k-1)+6=n_1+\cdots+n_k+6.
\]
If the endomorphism $\varphi$ satisfies $\alpha_1=\iota$,
then we omit the first coordinate $\pi_0(\alpha_1)$ and define
a bijection $\vartheta'$ between
$\varphi$ and the points of $K^{M-1}$.
We define two mappings $\xi:K^M\to K^N$
and $\eta:K^{M-1}\to K^N$ by
\[
\xi(u)=\pi(\vartheta^{-1}(u)),
\quad u\in K^M,
\]
\[
\eta(v)=\pi(\vartheta_1^{-1}(v)),
\quad v\in K^{M-1}.
\]
(Although $\vartheta$, $\vartheta'$, $\xi$ and $\eta$ depend on the
factorization $n=n_1\cdots n_k$, we shall use the same notation for
all factorizations.)

Clearly, the coordinates of the points $\xi(u)$ and $\eta(v)$ in
$K^N$ are polynomial functions of the coordinates of $u\in K^M$ and
$v\in K^{M-1}$, respectively. Hence $\xi$ and $\eta$ are morphisms.
Since the sets $\vartheta({\mathcal B}(n_1,\ldots,n_k))$ and
$\vartheta'({\mathcal C}(n_1,\ldots,n_k))$ are open subsets of $K^M$
and $K^{M-1}$ defined by inequalities which guarantee that
$\deg(h_i)=n_i$ and $\lambda$ is an affine automorphism, they are of
dimension $M$ and $M-1$, respectively. Their images  under $\xi$ and
$\eta$, respectively, are constructible subsets of $K^N$. Since the
dimension of a subset of $K^M$ and $K^{M-1}$ under a morphism cannot
be bigger that the dimension of the set itself, we obtain that
\[
\dim(\pi({\mathcal B}(n_1,\ldots,n_k)))
\leq \dim(\vartheta({\mathcal B}(n_1,\ldots,n_k)))=M,
\]
\[
\dim(\pi({\mathcal C}(n_1,\ldots,n_k)))
\leq \dim(\vartheta'({\mathcal C}(n_1,\ldots,n_k)))=M-1.
\]
The image $\pi({\mathcal A}^{(n)})$ of the set of automorphisms
of degree $n$ is the union
of $\xi({\mathcal B}(n_1,\ldots,n_k))$
and $\eta({\mathcal C}(n_1,\ldots,n_k))$
on all ordered factorizations
$n=n_1\cdots n_k$, $n_i\geq 2$.
Hence
\[
\dim(\pi({\mathcal A}^{(n)}))=\max\{\dim(\xi({\mathcal B}(n_1,\ldots,n_k))),
\dim(\eta({\mathcal C}(n_1,\ldots,n_k)))\}
\]
\[
\leq\max\{M=n_1+\cdots+n_k+6\mid n_1\cdots n_k=n\}.
\]
Using the inequality $n_1+n_2\leq n_1n_2$ for $n_1,n_2\geq 2$, we obtain that
$n_1+\cdots+n_k\leq n_1\cdots n_k=n$ and
$\dim(\pi({\mathcal A}_n))\leq n+6$. We shall complete the proof if we show
that $\dim(\pi({\mathcal B}(n)))=n+6$.

The elements of ${\mathcal B}(n)$ have the decomposition
$\varphi=(f,g)=\alpha\circ\beta\circ\lambda$, where
\[
\alpha=(y,x+ay),\quad a\in K,
\]
\[
\beta=(x+h_ny^n+\cdots+h_2y^2,y),\quad h_i\in K,\quad h_n\not=0,
\]
\[
\lambda=(a_1x+b_1y+c_1,a_2x+b_2y+c_2),\quad a_i,b_i,c_i\in K,\quad a_1b_2\not=a_2b_1.
\]
Hence
\[
f=a_1(h_n(x+ay)^n+\cdots+h_2(x+ay)^2)+a_1y+b_1(x+ay)+c_1,
\]
\[
g=a_2(h_n(x+ay)^n+\cdots+h_2(x+ay)^2)+a_2y+b_2(x+ay)+c_2.
\]

Recall that the coordinates of $K^N$ are the coefficients of the endomorphisms
$\varphi=(f,g)$ of degree $\leq n$ and the coordinates of $K^M=K^{n+6}$ are determined by
the automorphisms $\alpha$ and $\beta$ and the affine endomorphism $\lambda$.
Let
\[
\vartheta(\varphi)=(a,h_n,\ldots,h_2,a_1,b_1,c_1,a_2,b_2,c_2)\in K^{n+6}.
\]
Then $\varphi\in{\mathcal B}(n)$ if and only if
$h_n\not=0$, $a_1b_2\not=a_2b_1$.
Considering the coefficients $x^iy^j$ of the first polynomial $f$,
let $p_i$ be the coefficient of $x^i$, $i=0,1,2,\ldots,n$,
$q_1$ be the coefficient of $x^{n-1}y$ and $q_2$ be the coefficient of $y$.
For the second polynomial $g$,
we denote by $r_1,r_2,r_3$ the coefficients of $x,y$ and the constant term.
Hence, for the pair $\varphi=(f,g)$ the corresponding coordinates are
\[
p_i=a_1h_i,\quad i=2,\ldots,n,\quad p_1=b_1,\quad p_0=c_1,
\]
\[
q_1=na_1h_na,\quad q_2=a_1+b_1a,
\]
\[
r_1=b_2,\quad r_2=a_2+b_2a,\quad r_3=c_2.
\]
Consider the open subset $\mathcal U$
of $K^{n+6}$ defined by the inequalities
\[
h_n\not=0,\quad a_1b_2\not=a_2b_1,\quad a_1\not=0,\quad b_2\not=0.
\]
Clearly, $\vartheta^{-1}({\mathcal U})\subset{\mathcal B}(n)$.
Let ${\mathcal W}=\pi({\mathcal B}(n))$ be the image of
${\mathcal B}(n)$ in $K^N$ and let ${\mathcal O}({\mathcal W})$
be the algebra of regular functions of $\mathcal W$. The functions
$p_i,q_i,r_i$ are defined on the image $\xi({\mathcal U})$ of
$\mathcal U$ and satisfy the following conditions there:
\[
a=\frac{q_1}{np_n},\quad b_1=p_1,\quad c_1=p_0,\quad
a_1=q_2-b_1a=q_2-\frac{p_1q_1}{np_n},
\]
\[
h_i=\frac{p_i}{a_1}=\frac{np_ip_n}{np_nq_2-p_1q_1},\quad i=2,\ldots,n,
\]
\[
b_2=r_1,\quad c_2=r_3,\quad a_2=r_2-b_2a=r_2-\frac{q_1r_1}{np_n}.
\]
Since the coordinate functions
$a,h_n,\ldots,h_2,a_1,b_1,c_1,a_2,b_2,c_2$
of $\vartheta(\varphi)$ are algebraically independent
on $\mathcal U$ and
$\xi({\mathcal U})$ contains an open subset, we obtain that
the $n+6$ functions $p_0,p_1,\ldots,p_n,q_1,q_2,r_1,r_2,r_3$ are also algebraically
independent in ${\mathcal O}({\mathcal W})$. Hence
$\dim(\pi({\mathcal B}(n)))\geq n+6$ which completes the proof.
\end{proof}

By the above result, we are also able to determine the dimension of
the set of coordinates with fixed degree.

\begin{theorem}\label{dimension of coordinates}
Let ${C_n}$ be the set of all coordinates of degree $n$ of $K[x,y]$.
Then $\dim(C_n)=n+3$ for $n\geq 2$ and $\dim(C_1)=3$.
\end{theorem}
\begin{proof}According to the well-known theorem of Jung-van der Kulk
\cite{J, K},  for a fixed coordinate $f\in K[x,y]$ with $\deg(f)>1$,
two automorphisms $(f,g)$ and $(f,g_1)$ with $\deg(g)<\deg(f)$ and
$\deg(g_1)<\deg(f)$ if and only if $g_1=cg+d$ where $c\in K-\{0\}$,\
$d\in K$, hence $\dim(C_n)=\dim(B_n)-2$, where $B_n$ is the set of
automorphisms $(f,g)$ of degree $n$ with $\deg(f)=n>\deg(g)$. We can
now determine $\dim(B_n)$ as follows: tracing back to the proof of
Theorem \ref{dimension of automorphisms}, in the decomposition of
$(f,g)$,\ $\alpha_1$ is always the identity automorphism as
$\deg(f)>\deg(g)$. Hence $\dim(B_n)=\dim(\pi({\mathcal
A}^{(n)}))-1=n+5$. Therefore $\dim(C_n)=\dim(B_n)-2=n+3$.

When $n=1$, obviously the set of coordinates
$$C_1:=\{ax+by+c\
\mid a, b,  c\in K,\ (a, b)\ne (0, 0)\}$$ has dimension $3$.
\end{proof}

\begin{corollary}\label{dimension for all automorphisms}
The dimension of the set ${\mathcal A}_n$ of
automorphisms of $K[x,y]$ of degree $\leq n$ is equal to $n+6$ for all $n\geq 2$.
\end{corollary}
\begin{proof}
The corollary follows immediately from
Theorem \ref{dimension of automorphisms}
because
\[
{\mathcal A}_n=\bigcup_{i=1}^n{\mathcal A}^{(i)},
\quad
\dim({\mathcal A}_n)=\max\{\dim({\mathcal A}^{(i)})\mid i=1,\ldots,n\}=n+6.
\]
\end{proof}

\

\section{The case of free Nielsen -- Schreier algebras}

For a background on varieties of algebras with the Nielsen -- Schreier property
see e.g. \cite{MSY}. We shall use that if such a variety is defined by
a homogeneous (with respect to each variable) system of polynomial identities,
then the automorphisms of the finitely generated free algebras are tame.
We have the following analogue of Theorem
\ref{dimension of automorphisms}:

\begin{theorem}\label{automorphisms for free algebras}
Let $F(x,y)$ be the free $K$-algebra with two generators in a
Nielsen -- Schreier variety defined by a homogeneous system of
polynomial identities. Let $c_n$ be the dimension of all homogeneous
polynomials $u(x)$ in one variable of degree $n$ in $F(x,y)$ and
assume that $1=c_1=c_2\leq c_3\leq\cdots$.

{\rm (i)} The dimension of the set of automorphisms of degree $n$ of $F(x,y)$ is
\[
d_n=c_1+c_2+\cdots+c_n+4+2\varepsilon,
\]
where $\varepsilon=1$ for unitary algebras and $\varepsilon=0$ for nonunitary algebras.

{\rm (ii)} For the free nonassociative algebra $K\{x,y\}$ the
generating function of the sequence $d_n$ is
\[
d(t)=\sum_{n\geq 1}d_nt^n=\frac{1+4(2+\varepsilon)t-\sqrt{1-4t}}{2(1-t)}.
\]
\end{theorem}

\begin{proof}
(i) We use the following well known property of free algebras in
Nielsen -- Schreier varieties defined by homogeneous polynomial
identities. If several homogeneous elements in the free algebra
are algebraically dependent, then one of them is a polynomial of the others.
This implies that
$\text{Aut}(F(x,y))=A\ast_CB$, where $A$ is the affine group
if we consider unitary algebras and the general linear group when we
allow nonunitary algebras, $B$ is the group of triangular automorphisms and $C=A\cap B$.
Hence we have an analogue of Proposition \ref{canonical form of automorphisms}.
Following the main steps of the proof of Theorem \ref{dimension of automorphisms}
we obtain that the polynomials $h_n(x)$ of degree $n$ without constant and linear term
depend on $c_2+\cdots+c_n$ coordinates. Hence for a fixed factorization
$n=n_1\cdots n_k$ the dimension of the affine space $K^M$ of the endomorphisms
\[
\varphi=(f,g)=\alpha_1\circ\beta_1\circ\alpha_2\circ\beta_2\circ
\cdots\circ\alpha_k\circ\beta_k\circ\lambda
\]
is equal to
\[
M=k+\sum_{i=1}^k(c_2+\cdots+c_{n_i})+4+2\varepsilon
=\sum_{i=1}^k(c_1+c_2+\cdots+c_{n_i})+4+2\varepsilon
\]
because $c_1=1$.
Hence the dimension $d_n$ of the automorphisms of degree $n$ satisfies
\[
d_n\leq\max\{M=\sum_{i=1}^k(c_1+c_2+\cdots+c_{n_i})+4+2\varepsilon\}.
\]
Using the inequalities $c_1\leq c_2\leq c_3\leq\cdots$ we obtain that
\[
(c_1+c_2+\cdots+c_{n_1})+(c_1+c_2+\cdots+c_{n_2})\leq c_1+c_2+\cdots+c_{n_1+n_2}
\]
and, by induction,
\[
\sum_{i=1}^k(c_1+c_2+\cdots+c_{n_i})
\leq c_1+c_2+\cdots+c_n.
\]
As in Theorem \ref{dimension of automorphisms} we reach the equality
for the set of automorphisms
$\varphi=\alpha\circ\beta\circ\lambda$, where $\beta$ is a
triangular automorphism of degree $n$.

\

(ii) For the free nonassociative algebra, see e.g. \cite{H}, the number $c_n$
is equal to the Catalan number
\[
c_n=\frac{1}{n}\binom{2n-2}{n-1},\quad n\geq 1,
\]
the generating function of the Catalan numbers satisfies the quadratic equation
$c^2(t)-c(t)-t=0$ and has the presentation
\[
c(t)=\frac{1-\sqrt{1-4t}}{2}.
\]
To complete the proof we use that if the generating function of the sequence
$a_1,a_2,\ldots$ is $a(t)$, then the generating fucntion of the sequence
\[
b_n=a_1+a_2+\cdots+a_n,\quad n=1,2,\ldots,
\]
has the form
\[
b(t)=\frac{a(t)}{1-t}.\]\end{proof}

\

\section*{Acknowledgements}

The authors would like to thank Arnaud Bodin for bringing their
attention to the problem and for helpful comments and suggestions.


\begin{thebibliography}{ABC}

\bibitem [B]{B}
A. Bodin, {\em Private communication}, May 20, 2008.

\bibitem[BCW]{BCW}
H. Bass, E.H. Connell, D. Wright,
{\em The Jacobian conjecture: reduction of degree
and formal expansion of the inverse},
Bull. Amer. Math. Soc. (N.S.) {\bf 7} (1982), No. 2, 287-330.



\bibitem[Cz]{Cz}
A.J. Czerniakiewicz,
{\em Automorphisms of a free associative algebra of rank 2. I, II},
Trans. Amer. Math. Soc. {\bf 160} (1971), 393-401; {\bf 171} (1972), 309-315.

\bibitem[D]{D}
W. Dicks,
{\em A commutator test for two elements to generate the free algebra of rank two},
Bull. London Math. Soc. {\bf 14} (1982), No. 1, 48-51.

\bibitem[DY]{DY}
V. Drensky, J.-T. Yu,
{\em Automorphisms of polynomial algebras and Dirichlet series},
preprint, arXiv: 0806.0681.

\bibitem[F]{F}
J.-P. Furter,
{\em On the variety of automorphisms of the affine plane},
J. Algebra {\bf 195} (1997), No. 2, 604-623.

\bibitem[H]{H}
M. Hall, Jr.
{\em Combinatorial Theory}, Second edition, Wiley-Interscience Series in Discrete Mathematics, A Wiley-Interscience Publication,
John Wiley \& Sons, Inc., New York, 1986.

\bibitem[J]{J}
H.W.E. Jung,
{\em \"{U}ber ganze birationale Transformationen der Ebene},
J. Reine und Angew. Math. {\bf 184} (1942), 161-174.

\bibitem[K]{K}
W. van der Kulk,
{\em On polynomial rings in two variables},
Nieuw Archief voor Wiskunde (3) {\bf 1} (1953), 33-41.

\bibitem[ML]{ML}
L.G. Makar-Limanov,
{\em On automorphisms of
free algebra with two generators} (Russian),
Funk. Analiz i ego Prilozh. {\bf 4} (1970),  No. 3, 107-108.
Translation: Functional Anal. Appl. {\bf 4} (1970), 262-263.

\bibitem[MSY]{MSY}
A.A. Mikhalev, V. Shpilrain, J.-T. Yu,
{\em Combinatorial Methods.
Free Groups, Polynomials, and Free Algebras},
CMS Books in Mathematics/Ouvrages de Math\'ematiques de la SMC,
{\bf 19}, Springer-Verlag, New York, 2004.

\bibitem[M]{M}
T.T. Moh,
{\em On the Jacobian conjecture and the configurations of roots},
J. Reine Angew. Math. {\bf 340} (1983), 140-212.

\bibitem[Wi]{Wi}
H.S. Wilf, {\em  Generatingfunctionology}, Second edition, Academic
Press, Inc., Boston, MA, 1994. Third edition, A K Peters, Ltd.,
Wellesley, MA, 2006.

\bibitem[W]{W}
D.Wright,
{\em The amalgamated free product structure of
${\rm GL}\sb{2}(k[X\sb{1},\cdots,X\sb{n}])$
and the weak Jacobian theorem for two variables},
J. Pure Appl. Algebra {\bf 12} (1978), No. 3, 235-251.

\end{thebibliography}
\end{document}